\documentclass{amsart}
\usepackage{amssymb,latexsym}
\usepackage{amsmath}
\usepackage{amscd}
\usepackage{graphicx}
 \usepackage{color}
\usepackage{enumerate}
\numberwithin{equation}{section}
\theoremstyle{plain}
 \newtheorem{theorem}{Theorem}[section]

 \newtheorem{corollary}[theorem]{Corollary}
\theoremstyle{definition}

\newcommand \datum {18:45, December 5, 2016} 
\renewcommand \phi {\varphi}
\newcommand \paraH {H^{+1}}
\newcommand \range [1] {\textup{Range}(#1)}
\newcommand \uepsilon {u}
\newcommand \defiff {\overset{\textup{def}}{\iff} }
\newcommand \lexless{ <^{\textup{lex}}}
\newcommand \lexleq{ \leq^{\textup{lex}}}
\renewcommand \epsilon {{\boldsymbol\varepsilon}}
\newcommand \vvecr {\vec r\kern 1.5pt'}
\newcommand \ucirc {C_{\kern-1pt\textup{unit}}}
\newcommand \bnd [1] {\partial #1} 
\newcommand \lne {\ell}
\newcommand \flne[1] {{\ell^-_{#1}}}
\newcommand \llne[1] {{\ell^+_{#1}}}
\newcommand \fll[1] {{\ell_{#1}}}

\newcommand \slcr [1] {\textup{Sli}(#1)}
\newcommand \asli [2] {\textup{Sli}_{#1}(#2)}

\newcommand \cylin  {\textup{Cyl}}
\newcommand \dir [1] {\textup{dir}(#1)}
\newcommand \lder[1] {#1'_{-}}
\newcommand \rder[1] {#1'_{+}}
\newcommand \sder[1] {#1^{(\textup{sub})}}
\newcommand \nothing [1] {}
\newcommand\set [1]{\{#1\}}
\newcommand \tuple [1] {\langle #1 \rangle}
\newcommand \pair [2] {\tuple{#1,#2}}
\newcommand \real {\mathbb R}

\newcommand \preal {\real^{2}}
\newcommand \red [1] {{\color{red}#1\color{black}}}
\newcommand \tbf[1] {\textbf{#1}} 
\newcommand \dist[2]{\textup{dist}(#1,#2)}
\newcommand \mdist[2]{{\textup{d}}_{\textup{M}}(#1,#2)}

\newcommand \fmidp[1] {\textup{midp}}

\begin{document}
\title[Supporting lines of compact convex sets]
{A note and a short survey on supporting lines of compact convex sets in the plane}


\author[G.\ Cz\'edli]{G\'abor Cz\'edli}
\email{czedli@math.u-szeged.hu}
\urladdr{http://www.math.u-szeged.hu/\textasciitilde{}czedli/}
\address{University of Szeged\\ Bolyai Institute\\Szeged,
Aradi v\'ertan\'uk tere 1\\ Hungary 6720}

\author[L.\,L.\ Stach\'o]{L\'aszl\'o L.\ Stach\'o}
\email{stacho@math.u-szeged.hu}
\urladdr{http://www.math.u-szeged.hu/\textasciitilde{}stacho/}
\address{University of Szeged\\ Bolyai Institute\\Szeged,
Aradi v\'ertan\'uk tere 1\\ Hungary 6720}

\thanks{This research was supported by
NFSR of Hungary (OTKA), grant number K 115518}

\begin{abstract} After surveying some known properties of compact convex sets in the plane, 
we give a two rigorous proofs of the general feeling that supporting lines can be \emph{slide-turned} slowly and continuously. Targeting a wide readership, our treatment is elementary on purpose.
\end{abstract}




\date{\red{\datum}}

\maketitle


\section{Motivation}
Nowadays, there is a growing interest in the combinatorial properties of convex sets, usually, in compact convex sets. A large part of the papers belonging to this field go back to Erd\H os and Szekeres~\cite{erdosszekeres}; see, for example,
Dobbins, Holmsen, and Hubard~\cite{dobbinsatal2014} and \cite{dobbinsatal2015},  Pach and T\'oth~\cite{pachtoth1998} and \cite{pachtoth2000}, and their references. Recently, besides combinatorists and geometers, algebraists are also interested in compact convex sets; see, for example, Adaricheva~\cite{adaricarousel}, Adaricheva and Bolat~\cite{kabolat}, 
Adaricheva and Nation~\cite{kirajbbooksection}, Cz\'edli~\cite{czgcircles}, \cite{czgabp}, and \cite{czgcharcirc}, Cz\'edli and Kincses~\cite{czgkj}, and Richter and Rogers~\cite{richterrogers}. The interest of algebraists is explained by the fact that \emph{antimatroids}, introduced by 
Korte and Lov\'asz~\cite{kortelovasz81} and \cite{kortelovasz83}, and the dual concept of \emph{abstract convex geometries}, introduced by Edelman and Jamison~\cite{edeljam}, have close connections to lattice theory. These connections are surveyed in Adaricheva and Cz\'edli~\cite{kaczg}, Adaricheva and Nation~\cite{kirajbbooksection},  Cz\'edli~\cite{czgcoord}, and Monjardet~\cite{monjardet}. Finally, there are other types of combinatorial investigations of convex sets; the most recent is, perhaps, Novick~\cite{novick}.

One of the most important concepts related to planar convex sets is that of \emph{supporting lines}. Most of the papers mentioned above rely, explicitly or implicitly, on the properties of these lines. We guess that not only the experts of advanced analysis of convex sets and functions are interested in the above papers; at least, this is surely true in case of the first author of the present paper. However, it is quite difficult to explain to or understand by all the interested readers 
in a short, easy-to-follow, but rigorous way that why one of the most useful  property of compact convex sets holds. This property, which seems to be  absent in the literature, will be formulated in Theorem~\ref{thmmain}. This theorem is the ``note'' occurring in the title.

This motivates the aim of this short paper: even if Theorem~\ref{thmmain} could be proved in a shorter way by using advanced tools of Analysis and even if it states what is expected by geometric intuition, we are going to give a rigorous proof for it. Actually, we give two different proofs. We believe that if other statements for planar compact convex sets like \eqref{eqtxtsrctspthm} deserve proofs that are easy to reference, then so does this theorem. 
Note that Cz\'edli~\cite{czgcharcirc} exemplifies why the present paper is expected to be useful in further research: while the first version,  arXiv:1611.09331v1, of \cite{czgcharcirc} spends a dozen of pages on properties of supporting lines, its second version needs only few lines and a reference to the present paper.   Also, we 
exemplify the use of Theorem~\ref{thmmain} by an easy corollary, which is a well known but we have not found a rigorous proof for it.

\section{A short survey}
A \emph{compact} subset of the plane $\preal$ is a  topologically closed bounded subset.  The \emph{boundary} of $H$ will be denoted by $\bnd H$.   A subset $H$ of $\preal$ is \emph{convex}, if for any two points $X,Y\in\preal$, the closed line segment $[X,Y]$ is a subset of $H$. 
In this section, $H$ will stand for a compact convex set. Even if this is not always repeated, we always assume that a convex set is nonempty.
Each line $\lne$ gives rise to two \emph{closed halfplanes}; their intersection is $\lne$. Usually, unless otherwise is stated explicitly,  we assume that $\lne$ is a \emph{directed line}; then we can speak of the \emph{left and right halfplanes} determined by $\lne$. Points or set in the left halfplane are \emph{on the left} of $\lne$; being on the right is defined analogously. If $H$ is on the left of $\lne$ such that $H\cap\lne=\emptyset$, then $H$ is \emph{strictly on the left} of $\lne$.
The \emph{direction} of a directed line $\lne$ will be denoted by $\dir \lne\in[-\pi,\pi)$. 
It is understood modulo $2\pi$, whence we  could also consider $\dir\lne$ an element of $[0,2\pi)$. Furthermore, denoting the unit circle $\set{\pair x y: x^2+y^2=1}$ by $\ucirc$, we will often say that $\dir\lne\in \ucirc$. 
Following the convention of Yaglom and Boltyanski$\breve\i$~\cite{yaglomboltyanskii}, if $H$ is on the left of $\lne$ and $\lne\cap H\neq\emptyset$, then $\lne$ is a \emph{supporting line} of $H$. Clearly, for a supporting line $\lne$ of $H$, $\lne\cap H=\lne\cap\bnd H\neq\emptyset$.
We know from Yaglom and Boltyanski$\breve\i$~\cite[page 8]{yaglomboltyanskii} that parallel to each line $\lne$, a compact convex set with nonempty interior has  exactly two supporting lines.
Hence, without any stipulation on the interior,
\begin{equation}
\parbox{7cm}{for every $\alpha\in\ucirc$, a compact convex set has exactly one supporting line of direction $\alpha$.}
\label{eqtxtprhzMs}
\end{equation} 
Note at this point that, by definition, a \emph{curve} is the range $\range g$ of a continuous function $g$ from an interval $I$ of positive length to $\real^n$ for some $n\in\set{2,3,4,\dots}$. If $x_1\neq x_2\Rightarrow g(x_1)\neq g(x_2)$ except possibly for the endpoints of $I$, then $\range g$ is a \emph{simple curve}. 
A \emph{Jordan curve} is a homeomorphic \emph{planar}  image of a circle of nonzero radius, that is, a Jordan curve   is a \emph{simple closed curve in the plane}. A curve is \emph{rectifiable} if the lengths of its inscribed polygons form a bounded subset of $\real$.  The  following statement is known, say, from Latecki, Rosenfeld, and  Silverman~\cite[Thm.\ 32]{lateckiatal} and Topogonov~\cite[page 15]{topogonov}; see also \cite{wikialeph0}.
\begin{equation}
\parbox{7.6cm}{For a compact convex $H\subseteq \preal$ with nonempty interior, $\bnd H$ is a rectifiable Jordan curve.}
\label{eqtxtrectcont}
\end{equation}

For $P\in \bnd H$, there are  two possibilities; see, for example, Yaglom and Boltyanski$\breve\i${} \cite[page 12]{yaglomboltyanskii}. First, if there is exactly one supporting line through $P$, 
\begin{equation}
\text{then $P$ is a \emph{regular point} of $\bnd H$ and the  curve $\bnd H$ is  \emph{smooth} at $P$.}
\label{eqtxtmkrSm}
\end{equation}
Second, if there are at least two distinct supporting  lines $\lne_1$ and $\lne_2$ through $P$, then  $P$ is a \emph{corner} of $\bnd H$ (or of $H$). 
In both cases, a supporting line $\lne$ containing $P$ is called the  \emph{last semitangent} of $H$ through $P$ if for every small positive $\epsilon$,  there is an $\epsilon'\in(0,\epsilon)$ such that the line obtained from $\lne$ by rotating it around $P$ \emph{forward} (that is, counterclockwise) by $\epsilon'$ degree is not a supporting line. The \emph{first semitangent} is defined similarly. The first and the last semitangents coincide iff $P\in\bnd H$ is a regular point. For $P\in \bnd H$, 
\begin{equation}
\parbox{9.5cm}{$\flne P$ and $\llne P$ will denote the first semitangent and the last semitangent through $P$, respectively. When they coincide,  $\fll P:=\flne P=\llne P$ will stand for the \emph{tangent line} through $P$.}
\label{eqtxtfllllnE}
\end{equation}
Let us emphasize that no matter if $P\in \bnd H$ is a regular point or a vertex,  
\begin{equation}
\parbox{9cm}{there exists a supporting line through $P$; in particular, both $\flne P$ and $\llne P$ exist and they are uniquely determined.}
\label{eqtxtsupplineateachpoint}
\end{equation}
Besides Yaglom and Boltyanski$\breve\i$~\cite{yaglomboltyanskii}, this folkloric fact is also included, say, in  Boyd and Vanderberghe \cite[page 51]{boyd}. We note but will not use the fact that every line separating $P$ and the interior of $H$ is a supporting line through $P$. As an illustration for \eqref{eqtxtsupplineateachpoint}, some supporting lines of $H$ are given in Figure~\ref{figglide}. If $\lne_i$ is the supporting line denoted by $i$ in the figure, then $\lne_1=\lne_{P_1}$ is a tangent line,
$\flne{P_2}=\lne_2$ is the first semitangent through $P_2$, and $\llne{P_2}=\lne_4$ is the last semitangent through the same point.
We know from, say,  Borwein and Vanderwerff~\cite[2.2.15 in page 42]{borweinvandervefff},   Yaglom and Boltyanski$\breve\i$~\cite[page 110]{yaglomboltyanskii}, or even from \cite{wikialeph0}, that the boundary $\bnd H$ of a compact convex set $H\subseteq\preal$ can have $\aleph_0$ many corners. This possibility, which is not so easy to imagine, also justifies that we are going to give a rigorous proof for our theorem.
Next, restricting ourselves to the compact case and to the plane, we recall the \emph{strict separation theorem} as follows.
\begin{equation}
\parbox{8cm}{If $H_1,H_2\subseteq \preal$ are \emph{disjoint} compact convex set, then there exists a directed line $\lne$ such that $H_1$ is strictly on the left and $H_2$ is strictly on the right of $\lne$.}
\label{eqtxtsrctspthm}
\end{equation}
This result follows, for example, from Subsection 2.5.1 in 
Boyd and Vandenberghe~\cite{boyd} plus the fact that the \emph{distance} $\dist{H_1}{H_2}$ of $H_1$ and $H_2$ is positive in this case.

\begin{figure}[ht] 
\centerline
{\includegraphics[scale=1.0]{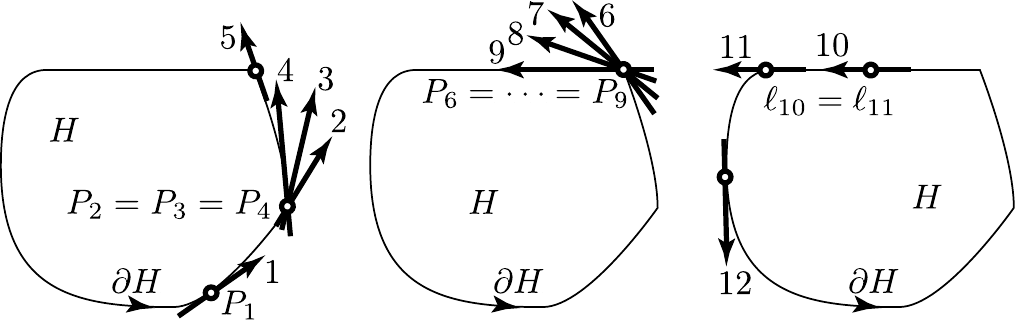}}
\caption{Supporting lines
\label{figglide}}
\end{figure}%

\section{A note and its corollary} Given a compact convex set $H$,  visual intuition tells us that any supporting line can be
continuously transformed to any other supporting line. We think of this transformation as a \emph{slow}, \emph{continuous} progression in time. For example, in Figure~\ref{figglide}, $\lne_{i+1}$ comes, after some time, later than $\lne_i$, for $i\in{1,\dots,11}$. While continuity makes a well-known mathematical sense, a comment on slowness is appropriate here. By \emph{slowness} we shall mean rectifiability, because this is what guarantees that running the process with a constant speed, it will terminate.  Therefore, since rectifiability is an adjective of curves, we are going to associate a simple closed rectifiable curve with $H$ such that the progression is described by moving along this curve forward. 
The only problem with this initial idea is that, say, $\lne_{11}$ cannot follow $\lne_{10}$, because they are the same supporting lines. Therefore, we consider pointed supporting lines. A \emph{pointed supporting line} of $H$ is a pair $\pair P\lne$ such that $P\in\bnd H$ and $\lne$ is a supporting line of $H$ through $P$. The transition from $\lne_i$ to $\lne_{i+1}$ will be called \emph{slide-turning}. Of course, the $\pair{P_i}{\lne_i}$, for $i\in\set{1,\dots,12}$, represent only twelve snapshots of a continuous progression.
In order to capture the progression mathematically, note that  each pointed supporting line  $\pair P\lne$ of $H$
is determined uniquely by the point $\pair P{\dir\lne}\in \real^4$. To be more precise, define the following~\emph{cylinder}
\begin{equation}
\cylin := \preal\times\ucirc =    \set{\tuple{x,y,z,t}\in\real^4: z^2+t^2=1} \subseteq \real^4.
\label{eqtxtcyldR}
\end{equation}
As the crucial concept of this section, the \emph{slide curve} of $H$ is  
\begin{equation}
\slcr H:=\{\pair P{\dir\lne} : \pair P\lne 
\text{ is a pointed supporting line of } H\};
\label{eqtxtsldcrV}
\end{equation}
it is a subset of $\cylin$. Although $\slcr H$ looks only a set at present, it will soon turn out that it is a curve. Actually, the main result of the paper says the following.

\begin{theorem}\label{thmmain} For every nonempty compact convex set $H\subseteq \preal$,
$\slcr H$ is a \emph{rectifiable} simple closed curve.
\end{theorem}

In order to exemplify the usefulness of this theorem, we state a  corollary. Although it is well known, we have not found a rigorous proof for it.

\begin{corollary}\label{corolfourtl} If $H_1,H_2\subseteq \preal$ are disjoint compact convex sets with nonempty interiors, then they have exactly four non-directed supporting lines in common.
\end{corollary}

The stipulation on the interior above can be relaxed but then we have to speak of \emph{directed} supporting lines.

\begin{figure}[ht] 
\centerline
{\includegraphics[scale=1.0]{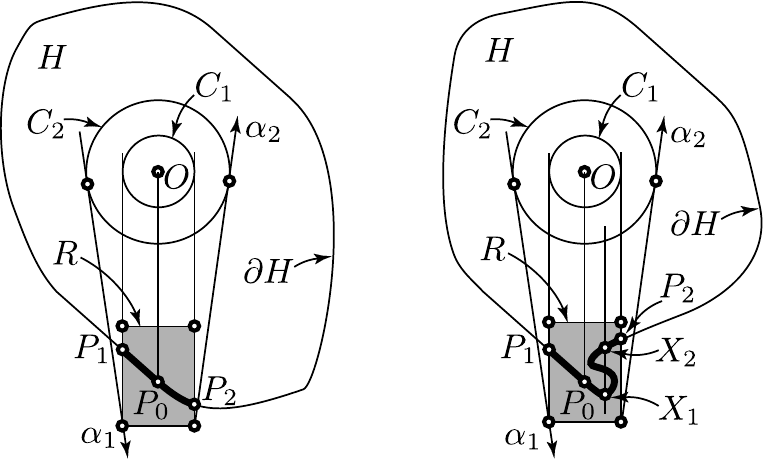}}
\caption{Reducing the problem to functions
\label{figrdtfncts}}
\end{figure}%

\section{Proofs}

\begin{proof}[First proof of Theorem~\ref{thmmain}]
We can assume that the interior of $H$ is nonempty, because otherwise $H$ is a line segment, possibly a singleton segment, and the statement trivially holds. 
In order to reduce the task to functions rather than convex sets, let $P_0$ be an arbitrary point of $\bnd H$. Pick a point $O$ in the interior of $H$, and choose a coordinate system such that both $P_0$ and $O$ are on the $y$-axis and $O$ is above $P_0$; see on the left of Figure~\ref{figrdtfncts}. For a positive $\uepsilon$, let $C_1$ and $C_2$ be the circles of radii $\uepsilon$ and $2\uepsilon$ around $O$; we can assume that $\uepsilon$ is so small that $C_2$ is in the interior of $H$. Let $A$  be the intersection of $\bnd H$ and  the closed strip $S$ between the two vertical tangent lines of $C_1$. In the figure, $A$ is the thick arc of $\bnd H$ between $P_1$ and $P_2$. Let 
\begin{equation}
\parbox{8.7cm}{$\asli A H:=\set{\pair P{\dir\lne}: P\in A,\,\,\,\pair P{\dir\lne}\in \slcr H}$, and similarly for future other arcs of $\bnd H$.}
\label{eqhsljH}
\end{equation}
Since the distance of $O$ and the complementer set of $H$ is positive, we can assume that $\uepsilon$ is so small that the grey-filled rectangle containing $A$ in the figure is strictly below $C_2$. (We have some freedom to choose the upper and lower edges of this rectangle.)
Let $\alpha_1,\alpha_2\in\ucirc$ be the directions of the external common supporting lines of $C_2$ and this rectangle, see the figure. Note that if we consider $\ucirc$ the interval $[-\pi,\pi)$, then $\alpha_1=-\alpha_2$.   The presence of $C_2$ within $H$ guarantees the second half of the following observation:
\begin{equation}
\parbox{7.7cm}{$0<\alpha_2<\pi$ and for every supporting line $\lne$ of $H$ that contains a point of $A$, $-\alpha_2\leq \dir\lne\leq \alpha_2$.}
\end{equation}
We claim that 
\begin{equation}
\text{$A$ is the graph of a convex function $f\colon [-\uepsilon,\uepsilon]\to \real$.}
\label{eqtxtfsdzbnsT}
\end{equation}
By the convexity of $H$ and \eqref{eqtxtrectcont}, every vertical line in the strip $S$ intersects $A$. Suppose, for a contradiction, that $U$ is not the graph of a function.  Then a vertical line in $S$ intersects $A$ in at least two distinct points, $X_1$ and $X_2$. Let, say, $X_2$ be above $X_1$;  see on the right of the figure. Then $X_2$ is in the interior of the convex hull of $\set{X_1} \cup C_2$, whereby it is in the interior rather than on the boundary of $H$. 
This contradiction shows that  $f$ is a function. It is convex, since so is $H$. This  proves \eqref{eqtxtfsdzbnsT}. 
Clearly, the same consideration shows that 
\begin{equation}
\text{each ray starting from $O$ intersects $\bnd H$ exactly once.}
\label{eqtxtchrYsntrXnC}
\end{equation}

Next, recall from, say,  Niculescu and Persson \cite[page 25]{niculescup}  that for a real-valued function $f\colon \real \to \real$ and $x_0$ in the interior of its domain, the left derivative $\lim_{x\to x_0-}(f(x)-f(x_0))/(x-x_0)$ and the right derivative of $f$ at $x_0$ are denoted by $\lder f(x_0)$ and $\rder f(x_0)$, respectively. By a theorem of Stolz~\cite{stolz}, see also  Niculescu and Persson~\cite[Theorem 1.3.3]{niculescup}, if $f$ is convex in the open interval $(-u,u)$, then 
\begin{equation}
\parbox{10cm}{for all $x,x_1,x_2\in (-\uepsilon,\uepsilon)$, both $\lder f(x)$ and $\rder f(x)$ exist, \\$\lder f(x)\leq \rder f(x)$, and $x_1<x_2$ implies that $\rder f(x_1)\leq \lder f(x_2)$.}
\label{eqtxtdmzBmQk}
\end{equation}
Recall that a function $g$ from a subset of $\real^k$ to $\real^n$  is a \emph{Lipschitzian} if there exists a positive constant $L$ such that $\dist{g(x)}{g(x')}\leq L\cdot \dist{x}{x'}$ holds for all $x$ and $x'$ in the domain of $g$. Since $f$ is convex, we know from 
Rockafellar~\cite[Theorems 10.1, 10.4, and 24.1]{rockafellar} 
that 
\begin{equation}
\parbox{7.5cm}{in $(-u,u)$, $f$ is Lipschitzian, $\lder f$ is continuous from the left, and $\rder f$ is continuous from the right.}
\label{eqtxtlfrGhTcns}
\end{equation}
Note that if a function is Lipschitzian in an interval, then it is uniformly continuous there. 
From now on, we consider $f$ only in the  open interval $(-\uepsilon,\uepsilon)$ and  we fix a positive $v\in (0,\uepsilon)$,  For  $x_0\in (-\uepsilon,\uepsilon)$, the \emph{subdifferential} is defined as the interval
\begin{equation}
\begin{aligned}
\sder f(x_0)&=\set{d\in\real: \forall x\in (-
\uepsilon,\uepsilon),\,\, f(x)\geq f(x_0)+d(x-x_0)} \cr
& = [\lder f(x_0), \rder f(x_0)];
\end{aligned}
\label{eqdFsubDiFf}
\end{equation}
see Niculescu and Persson \cite[Section 1.5]{niculescup}. 
As a consequence of \eqref{eqtxtdmzBmQk}, the subdifferential is a \emph{dissipative}  set-valued function, that is,
\begin{equation}
\parbox{7.5cm}{for $x_1,x_2\in (-\uepsilon,\uepsilon)$, if $x_1<x_2$, $d_1\in \sder f(x_1)$, and $d_2\in \sder f(x_2)$, then $d_1\leq d_2$.}
\label{eqtxtdSSptv}
\end{equation} 
Consider the set 
\begin{equation}
D:=\set{\pair x d:  x\in [-v,v]\text{ and }d\in\sder f(x)}\subseteq \preal
\label{eqDstdFnTz}
\end{equation}
with the (strict) lexicographic ordering
\begin{equation} \pair {x_1}{d_1}\lexless \pair {x_2}{d_2}\defiff 
(x_1<x_2, \text{ or } x_1=x_2\text{ and }d_1<d_2).
\label{eqlexDeFn}
\end{equation}
We define a function
\begin{equation}
t\colon D\to \real\,\,\text{ by }\,\,t(x, d)=  x+d.
\label{eqtxtttDfCsfbt}
\end{equation}
Note that  $t(x,d)$ is a short form of  $t(\pair x d)$. Recall that 
the \emph{Manhattan distance} $\mdist{\pair {x_1}{d_1}}{\pair {x_2}{d_2}}$ of $\pair {x_1}{d_1}$  and $\pair {x_2}{d_2}$ in $\preal$ is defined as $|x_1-x_2|+|d_1-d_2|$. It has the usual properties of a distance function.  It follows from \eqref{eqtxtdmzBmQk} that, for  $\pair {x_1}{d_1}$ and $\pair {x_2}{d_2}$ in $D$ (rather than in $\preal$),
\begin{equation}
\text{if }\pair {x_1}{d_1} \lexleq  \pair {x_2}{d_2},\text{ then }
\mdist{\pair {x_1}{d_1}}{\pair {x_2}{d_2}}= t(x_2,d_2)-t(x_1,d_1);
\label{eqtxtFThNdhW}
\end{equation}
that is, for points of $D$, the Manhattan distance is derived from the function $t$. Let $\dist{\pair {x_1}{d_1}}{\pair {x_2}{d_2}}$ stand for the Euclidean distance 
$((x_1-x_2)^2+(d_1-d_2)^2))^{1/2}$; in $\real^4$, it is understood analogously. 
For the sake of a later reference, we note in advance that for $x^{(i)},d^{(i)}\in\preal$, the \emph{Manhattan distance} in $\real^4$ is understood as 
\begin{equation}
\mdist {\pair{x^{(1)}}{d^{(1)}}} {\pair{x^{(2)}}{d^{(2)}}}:=
\dist {x^{(1)}}{x^{(2)}} + \dist {d^{(1)}}{d^{(2)}}. 
\label{eqrngmnHdSt}
\end{equation}
It is well known and easy to see that, for all $\pair {x_1}{d_1}, \pair {x_1}{d_1}$ in $\preal$, and even  $\real^4$ if $x_1, x_2,d_1,d_2\in\preal$,
\begin{equation}
\dist{\pair {x_1}{d_1}}{\pair {x_2}{d_2}}
\leq \mdist{\pair {x_1}{d_1}}{\pair {x_2}{d_2}}\leq2\cdot\dist{\pair {x_1}{d_1}}{\pair {x_2}{d_2}}.
\label{eqdsmDsDsT}
\end{equation}
It follows from \eqref{eqtxtFThNdhW} and the second half of \eqref{eqdsmDsDsT} that $t$ is a Lipschitzian function (with Lipschitz constant $2$). Since $\mdist--$ is a distance function, \eqref{eqtxtFThNdhW} yields that $t$ is injective.  Actually, it is bijective as a $D\to\range t$ function. Thus, it has an inverse function, $t^{-1}\colon\range t\to D$, which is also bijective. In order to see that the function $t^{-1}$ is also a Lipschitzian, let $y_i=t(x_i,d_i)=x_i+d_i\in \range t$, for $i\in\set{1,2}$. Since $\dist --$ is a symmetric function, we can assume that  $\pair {x_1}{d_1}\lexleq\pair {x_2}{d_2}$. We can also assume that $d_1\leq d_2$; either because $x_1=x_2$ and then we can interchange the subscripts 1 and 2, or because $x_1<x_2$ and  \eqref{eqtxtdSSptv} applies. With these assumptions, let us compute:
\begin{align*}\dist{y_1}{y_2}&=|y_2-y_1|=|x_2+d_2-(x_1+d_1)|= |x_2-x_1+d_2-d_1|\cr
&= x_2-x_1+d_2-d_1 = |x_1-x_2|+|d_1-d_2|= \mdist{\pair {x_1}{d_1}}{\pair {x_2}{d_2}}.
\end{align*}
Hence, using the second part of \eqref{eqdsmDsDsT}, it follows that the function $t^{-1}$ is  Lipschitzian (with Lipschitz constant 2). So, we can summarize that 
\begin{equation}
\parbox{10.5cm}{$t\colon D\to\range t$ and $t^{-1}\colon \range t\to D$ are reciprocal bijections and both of them are Lipschitzian; in short, $t$ is \emph{bi-Lipschitzian}.}
\label{eqtxtsTrpRcL}
\end{equation}
Next, let $w_1=t(-v, \lder f(-v))$ and $w_2=t(v, \rder f(v))$. We claim that
\begin{equation}
\range t= [w_1,w_2].
\label{eqWhTrnGt}
\end{equation}
In order to see the easier inclusion, assume that $\pair x d\in D$. Using \eqref{eqtxtdSSptv} and \eqref{eqlexDeFn},
we obtain that $\pair{-v}{\lder f(-v)}\lexleq \pair x d \lexleq\pair{v}{\rder f(v)}$. Thus, since \eqref{eqtxtFThNdhW} yields that $t$ is order-preserving, we conclude that $w_1\leq t(x,d)\leq w_2$, that is, $\range t \subseteq [w_1,w_2]$. In order to show the converse inclusion, assume that $s\in[w_1,w_2]$. We need to find an $\pair {x_0}{d_0}\in D$ such that $s=t(x_0,d_0)$, that is, $s=x_0+d_0$. Define 
\begin{equation}
\parbox{10.3cm}{$x^-:=\sup\,\set{x: \text{there is a } d\text{ such that }\pair x d\in D\text{ and } x+d\leq s }$,\\ 
$x^+:=\inf\,\set{x: \text{there is a } d\text{ such that }\pair x d\in D\text{ and } x+d\geq s }$.}
\label{eqtxtinfsupXmp}
\end{equation}
Since $t(-v,\lder f(-v))=w_1\leq s  \leq w_2 = t(v,\rder f(v))$, the sets occurring in \eqref{eqtxtinfsupXmp} are nonempty. Hence, both $x^-$ and $x^+$ exist and we have that $x^-, x^+\in [-v, v]$. Suppose, for a contradiction, that $x^+<x^-$. Then  $x^-=3\epsilon+ x^+$ for a positive $\epsilon$.    By \eqref{eqtxtinfsupXmp}, which defines $x^-$ and $x^+$,   we can pick $\pair {x^\dagger}{d^\dagger}, \pair {x^\ddagger}{d^\ddagger}\in D$ such that 
${x^\dagger}\in (- \epsilon + x^-,  x^-]$, $t({x^\dagger},{d^\dagger}) ={x^\dagger}+{d^\dagger} \leq s$, 
${x^\ddagger}\in [x^+, \epsilon + x^+ )$,  and $t({x^\ddagger},{d^\ddagger}) ={x^\ddagger}+{d^\ddagger} \geq s$. 
In particular, ${x^\dagger} + {d^\dagger}\leq {x^\ddagger}+{d^\ddagger}$. 
However, since $x^\ddagger < x^\dagger$, the dissipative property from \eqref{eqtxtdSSptv} gives that  $d^\ddagger\leq d^\dagger$, whereby 
${x^\dagger} + {d^\dagger}\geq {x^\dagger}+{d^\ddagger}>{x^\ddagger}+{d^\ddagger}$, contradicting  ${x^\dagger} + {d^\dagger}\leq {x^\ddagger}+{d^\ddagger}$. This proves that 
$x^-\leq x^+$. 
Next, suppose for a contradiction that $x^-<x^+$. Let  $x^\ast:=(x^-+x^+)/2$, and pick a $d^\ast\in\sder f(x^\ast)$. 
Since $x^\ast + d^\ast \leq s$ would contradict the definition of $x^-$, we have that $x^\ast + d^\ast >s$, which contradicts the definition of $x^+$.  This excludes the case $x^-<x^+$. So we have that  $x^-=x^+$, and we let $x_0:= x^-=x^+$. Clearly, for all $x$ and the corresponding $d$ in the upper line of \eqref{eqtxtinfsupXmp}, $x+\lder f(x)\leq x+d\leq s$. Hence, the left continuity formulated in \eqref{eqtxtlfrGhTcns} gives that $t(x_0,\lder f(x_0))=  x_0+\lder f(x_0)=x^- +\lder f(x^-)\leq s$. Similarly,  $t(x_0,\rder f(x_0))=  x_0+\rder f(x_0)=x^+ +\rder f(x^+)\geq s$. So $x_0+\lder f(x_0)\leq s 
\leq x_0+\rder f(x_0)$, whereby \eqref{eqdFsubDiFf} gives a $d_0\in[\lder f(x_0),\rder f(x_0)]$ such that  $s=x_0+d_0=t(x_0,d_0)$. This proves \eqref{eqWhTrnGt}.

It is well known (and evident) that, with self-explanatory domains,  
\begin{equation} 
\parbox{8.5cm}{the composition of two bi-Lipschitzian functions is bi-Lipschitzian. Thus, a bi-Lipschitzian function maps a rectifiable simple curve to a rectifiable simple curve.}
\label{eqtxtcmptWblPsH}
\end{equation}
Before utilizing \eqref{eqtxtcmptWblPsH}, we need some preparations.  Let $Q_1=\pair{-v}{f(-v)}$ and $Q_2=\pair{v}{f(v)}$; they are points on the arc $A$ before and after $P_0$, respectively. Let $B$ be the sub-arc of $A$ 
(and of $\bnd H$) from $Q_1$ to $Q_2$, and note that $P_0$ is in the interior of $B$. Let $f^\ast\colon [-v,v]\to B$ be the function defined by $f^\ast(x):=\pair x{f(x)}$.
Using \eqref{eqtxtlfrGhTcns} and the relation between  the Euclidean and the Manhattan distance functions, see \eqref{eqdsmDsDsT}, it follows that $f^\ast$ is Lipschitzian. This fact implies trivially that $f^\ast$ is bi-Lipschitzian.
So is the arctangent function on $[-v,v]$. Therefore, it follows in a straightforward way from \eqref{eqdsmDsDsT} that the Cartesian (or categorical) product function
\begin{equation}
\parbox{11cm}{
$\pair {f^\ast}{\arctan}\colon D\to \asli B H$,  defined by $\pair x d\mapsto \pair {f^\ast(x)}{\arctan(d)}$,
where $\asli B H$ is defined in \eqref{eqhsljH}, is bi-Lipschitzian.}
\label{eqtxtblIpSchTzNvj}
\end{equation}

The line segment $[w_1,w_2]$ is clearly a simple rectifiable  curve. So is $D$ by \eqref{eqtxtttDfCsfbt}, \eqref{eqtxtsTrpRcL},  \eqref{eqWhTrnGt}, and \eqref{eqtxtcmptWblPsH}. Hence,  \eqref{eqtxtcmptWblPsH} and \eqref
{eqtxtblIpSchTzNvj} yield that  $\asli B H$ is a simple rectifiable curve. Finally, since $P_0\in \bnd H$ was arbitrary and since the endpoints of $B$ can be omitted from $B$, we obtain that $\bnd H$ can be covered by a set  $\set{B_i: i\in I}$ of open arcs such that the $\asli{B_i}H\subseteq \cylin$ are simple rectifiable curves. Clearly, the $\asli{B_i}H$ cover $\slcr H$.  Since $\bnd H$ is compact, we can assume that $I$ is finite. Therefore, $\slcr H$ is covered by finitely many open simple rectifiable curves. Furthermore, 
\eqref{eqtxtchrYsntrXnC} yields that each of these open curves overlaps with its neighbors.
 Thus, we conclude  the validity of Theorem~\ref{thmmain}.
\end{proof}

In the following proof, the argument leading to \eqref{eqtxtPrHsrCtfTbL} can be extracted from the more general approach of Kneser~\cite{kneser} and Stach\'o~\cite{stacho}. For the planar case and for the reader's convenience, it is more convenient to prove  \eqref{eqtxtPrHsrCtfTbL} directly.

%

\begin{proof}[Second proof of Theorem~\ref{thmmain}]
Define $\paraH:=\set{P\in\preal: \dist P H)\leq 1}$. First, we prove that $\paraH$ is a compact convex set. Let $Q$ be a limit point of $\paraH$ and suppose, for a contradiction,  that $Q\notin \paraH$. This means that  $\dist Q H=1+3\epsilon$ for a positive $\epsilon\in \real$. Take a sequence $(P_n:n\in\mathbb N)$ of points in $\paraH$ such that $\lim_{n\to\infty} P_n=P$. For each $n\in N$, pick a point $Q_n\in H$ such that $\dist{P_n}{Q_n}\leq 1$. Since $H$ is compact, the sequence $(Q_n: n\in\mathbb N)$ has a convergent subsequence. Deleting members if necessary, we can assume that $(Q_n: n\in\mathbb N)$ itself  converges to a point $Q$ of $H$. Take a sufficiently large $n\in \mathbb N$ such that $\dist P{P_n}<\epsilon$ and $\dist {Q_n}Q<\epsilon$. Then $1+3\epsilon = \dist P Q\leq 
\dist P{P_n}+ \dist{P_n}{Q_n} + \dist {Q_n}Q \leq \epsilon + 1 + \epsilon =1+2\epsilon$ is a contradiction. Hence, $\paraH$ is closed, whereby it is obviously compact. In order to show that it is convex, let $X,Y\in \paraH$ and let $\lambda\in (0,1)$; we need to show that $Z:=(1-\lambda)X+\lambda Y\in \paraH$. The containments $X \in \paraH$ and $Y \in \paraH$ are witnessed by some $X_0,Y_0\in H$ such that $\dist X{X_0}\leq 1$ and $\dist Y{Y_0}\leq 1$. 
Since $H$ is convex, $Z_0:=(1-\lambda)X_0 + \lambda Y_0\in H$. The vectors $\vec a:=X-X_0$ and $\vec b:=Y-Y_0$ are of length at most 1, and it suffices to show that so is $\vec c:=Z-Z_0$.  Since $(\vec a,\vec b)\leq ||\vec a||\cdot||\vec b||\leq 1$, we have that
\begin{align*}
(\vec c,\vec c)&=((1-\lambda)\vec a +\lambda\vec b,(1-\lambda)\vec a +\lambda\vec b  )\cr
&=(1-\lambda)^2(\vec a,\vec a) + \lambda^2(\vec b,\vec b) + 2\lambda(1-\lambda) (\vec a,\vec b)\cr
&\leq (1-\lambda)^2  + \lambda^2 + 2\lambda(1-\lambda) =1.
\end{align*}
Hence, $\dist Z{Z_0}=||\vec c||\leq 1$, and $\paraH$ is convex. Thus,  \eqref{eqtxtrectcont} gives that
\begin{equation}
\bnd{\paraH} \text{ is rectifiable Jordan  curve}.
\label{eqtxtPrHsrCtfTbL}
\end{equation}

\begin{figure}[ht] 
\centerline
{\includegraphics[scale=1.0]{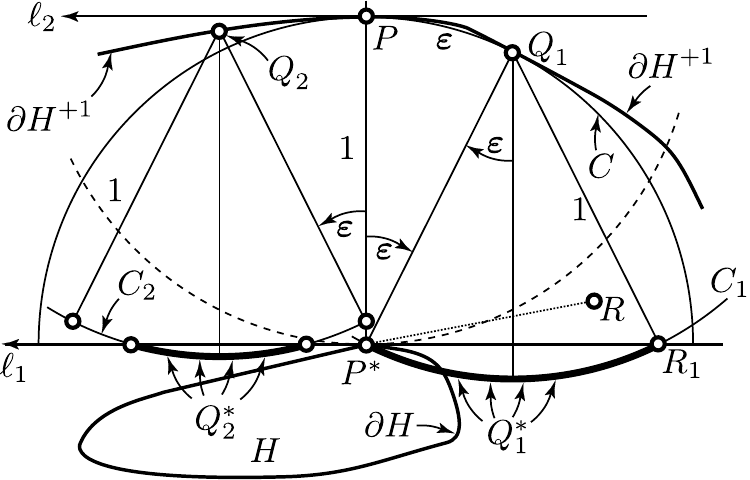}}
\caption{Illustration for the second proof}
\label{figPrHndHbnD}
\end{figure}%

Clearly, $\bnd{\paraH}=\set{X: \dist X H=1}=\set{X: \dist X {\bnd H}=1}$. Define the following relation 
\[\rho:= \set{\pair P{P^\ast}\in \bnd{\paraH}\times \bnd H: \dist P{P^\ast}=1}
\]
between $\bnd{\paraH}$ and $\bnd H$; see Figure~\ref{figPrHndHbnD}. Let $\pair P{P^\ast}\in\rho$ as in the figure. The coordinate system is chosen so that $P$ and $P^\ast$ determine a vertical line and $P^\ast$ is above $P$. Through $P^\ast$ and $P$, let $\lne_1$ and $\lne_2$ be the lines of direction $\pi$; they are perpendicular to $[P^\ast,P]$. We claim that
\begin{equation}
\text{$\lne_1$ is a supporting line of $H$.}
\label{eqtxtsloRbgnPk}
\end{equation}
Suppose to the contrary that $\lne_1$ is not a supporting line and pick a point $R\in H$ strictly on the right of $\lne_1$; see the figure. Since $P\in\bnd{\paraH}$, $\dist P H=1$, whereby $R$ cannot be inside the dotted circle of radius 1 around $P$. However, since this circle touches $\lne_1$ at $P^\ast$, the line segment $[P^\ast,R]$, which is a subset of $H$ by convexity, has a point inside the dotted circle. This contradicts $\dist P H=1$ and proves \eqref{eqtxtsloRbgnPk}. From  \eqref{eqtxtsloRbgnPk}, it follows that if $\pair P Q\in \rho$, then $Q=P^\ast$. Hence, 
\begin{equation*}
\parbox{10.5 cm}{$f\colon \bnd{\paraH}\to \slcr H$, defined by $f(P)=\pair{P^\ast}{\dir{\lne^\ast}}\in \slcr H \iff$\\
$\pair P {P^\ast}\in \rho$,  $\lne^\ast$ is a supporting line, and $\lne^\ast$ is perpendicular to $[P,P^\ast]$}
\end{equation*}
is a mapping. Trivially, 
\begin{equation}
\parbox{10 cm}{$g\colon\slcr H\to   \bnd{\paraH}$, defined by $g(\pair{P^\ast}{\dir{\lne^\ast}})=P \iff$  \\
$\dir{[P^\ast,P]}=\dir{\lne^\ast}-\pi/2$ and $\dist P{P^\ast}=1$,}
\label{eqtxtgmPdfktsrbN}
\end{equation}
is also a mapping. Moreover $f$ and $g$ are reciprocal bijections. Recall from  Luuk\-kainen~\cite[Definition 2.14]{luukk} that a function $\tau\colon X\to Y$ is \emph{Lipschitz in the small} if there are $\delta>0$ and $L\geq 0$ such that $\dist {\tau(x_1)}{\tau(x_2)}\leq L\cdot\dist{x_1}{x_2}$ for all $x_1,x_2\in X$ with $\dist{x_1}{x_2}\leq\delta$.
We know from \cite[2.15]{luukk} that every bounded function with this property is Lipschitzian. We are going to show that $f$ and $g$ are Lipschitz in the small, witnessed by  $\delta=1/5$ and $L=9$, because then $g=f^{-1}$, \eqref{eqtxtcmptWblPsH}, and  \eqref{eqtxtPrHsrCtfTbL}
will imply the theorem. (Note that $\delta=1/5$ and $L=9$ are convenient but none of them is optimal.)

First, we deal with $f$. Assume that $Q_1\in\bnd{\paraH}$ such that $\gamma:=\dist P{Q_1}<\delta=1/5$; see Figure~\ref{figPrHndHbnD}. The angle $\epsilon:=\angle(P P^\ast Q_1)$, which is the length of the circular arc from $P$ to $Q_1$, is close to $\gamma$ in the sense that 
\begin{equation}
\text{ both $\epsilon/\gamma$ and $\gamma/\epsilon $ are in the interval $(99/100, 101/100)$;}
\label{eqtxtkkszdszgyszd}
\end{equation}
this is shown by easy trigonometry since both $\sin(1/5)/(1/5)$ and $(1/5/)\sin(1/5)$ are in the open interval on the right of \eqref{eqtxtkkszdszgyszd}.
Let $C$ and $C_1$ be the circles of radius 1 around  $P^\ast$ and $Q_1$, respectively. 
Since $\dist{Q_1}H=1$, $Q_1$ is not in the interior of (the disk determined by) $C$. Also, since $\lne_1$ is a supporting line of $H$, we have that $\lne_2$ is a supporting line of $\paraH$ and $Q_1$ cannot be strictly on the right (that is, above) $\lne_2$. So either $Q_1$ is on the circle $C$, or it is above $C$ but not above $\lne_2$ (but then we write $Q_2$ instead of $Q_1$ in the figure). 
Denote $f(Q_1)$ by $\pair{Q_1^\ast}{\dir{\lne_1^\ast}}$.
Clearly, $Q_1^\ast$ is on the thick arc of $C_1$ from $P^\ast$ to $R_1$, as indicated in the figure. The length of this arc is $2\epsilon$, whence $\dist{P^\ast}{Q_1^\ast}\leq 2\epsilon$. Since $\lne_1^\ast$ is perpendicular to $[Q_1,Q_1^\ast]$ and $Q_1^\ast$ is on the thick arc of $C_1$, we have that $\dist{\dir{\lne^\ast}} {\dir{\lne_1^\ast}}\leq \epsilon\leq 2\epsilon$. So the Manhattan distance 
$\mdist{\pair{P^\ast}{\dir{\lne^\ast}}}{\pair{Q_1^\ast} {\dir{\lne_1^\ast}}}$,
see \eqref{eqrngmnHdSt}, 
is at most $4\epsilon$. Hence, \eqref{eqdsmDsDsT} and \eqref{eqtxtkkszdszgyszd} yield that $\dist{f(P)}{f(Q_1}\le 9\cdot \dist{P}{Q_1}$. The other case, represented by $Q_2$, follows from the fact that 
$\dist{P^\ast}{Q_2^\ast}$ and $\dist{\dir{\lne^\ast}} {\dir{\lne_2^\ast}}$ are smaller than the respective distances in the previous case. This shows that $f$ is Lipschitz in the small. 


Next, we deal with $g$. Assume that $\pair{P^\ast}{\dir{\lne^\ast}}$ and $\pair{P_1^\ast}{\dir{\lne_1^\ast}}$ are in $\slcr H$ and their distance, $\gamma$, is less than $\delta$. With the auxiliary point  $\pair{P^\ast}{\dir{\lne_1^\ast}}\in \real^4$, which need not be in $\slcr H$, we have that  $\dist{\pair{P^\ast}{\dir{\lne^\ast}}} {\pair{P^\ast}{\dir{\lne_1^\ast}}}\leq \gamma$ and 
 $\dist{\pair{P^\ast}{\dir{\lne_1^\ast}}} {\pair{P_1^\ast}{\dir{\lne_1^\ast}}}\leq \gamma$. Although the auxiliary point is not in the domain of $g$ in general, we can extend the domain of $g$ to this point by \eqref{eqtxtgmPdfktsrbN}. Since 
the secants of the unit circles are shorter than the corresponding circular arcs, whose length equals the corresponding central angles, it follows that 
$\dist{g(\pair{P^\ast}{\dir{\lne^\ast}})} {g(\pair{P^\ast}{\dir{\lne_1^\ast}})} \leq \gamma$.
Since parallel shifts are distance-preserving, $\dist{g(\pair{P^\ast}{\dir{\lne_1^\ast}})} {g(\pair{P_1^\ast}{\dir{\lne_1^\ast}})}= \gamma$. Hence, the triangle inequality yields that $\dist{g(\pair{P^\ast}{\dir{\lne^\ast}})} {g(\pair{P_1^\ast}{\dir{\lne_1^\ast}})}\leq 2\gamma\leq 9\delta$. Thus, $g$ is also Lipschitz in the small, as required. This completes the second proof of Theorem~\ref{thmmain}.
\end{proof}

\begin{figure}[ht] 
\centerline
{\includegraphics[scale=1.0]{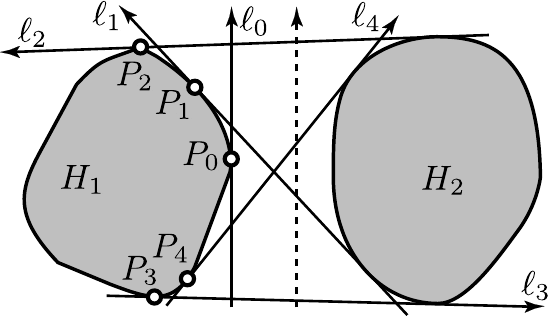}}
\caption{Illustration for Corollary~\ref{corolfourtl}
\label{figfkzTn}}
\end{figure}%

\begin{proof}[Proof of Corollary~\ref{corolfourtl}]
By \eqref{eqtxtsrctspthm}, we have a directed line, the dotted one in Figure~\ref{figfkzTn}, such that $H_1$ is strictly in the left and $H_2$ is strictly on the right of this line. By \eqref{eqtxtprhzMs}, we can take a $\pair{P_0}{\dir{\lne_0}}\in \slcr {H_1}$ such that $\lne_0$ and the dotted line have the same direction. 
For $0<L\in\real$, let 
\[\text{$L\cdot\ucirc$ denote the circle $\set{\pair x y: x^2+y^2=(L/(2\pi))^2}$ of perimeter $L$.}
\] 
Since $\slcr {H_1}$ is a rectifiable simple closed curve by Theorem~\ref{thmmain}, we can let $L$ be its perimeter. Let  
\begin{equation}
\text{$\set{h(t): t\in L\cdot \ucirc}$ be a parameterization of $\slcr{H_1}$}
\label{eqtxtslHpRmHTsVTnM}
\end{equation}
such that $\pair{P_0}{\dir{\lne_0}}=h(t_0)$. We think of the parameter $t$ as the \emph{time} measured in seconds. While the time $t$ is slowly passing, $\pair{P(t)}{\dir{\lne(t)}}$ is slowly and continuously moving forward along $\slcr{H_1}$, and  the directed supporting line $\pair{P(t)}{{\lne(t)}}$ is \emph{slide-turning} forward, slowly and continuously. Since $H_2$ is compact,
the distance $\dist{\lne(t)}{H_2}$ is always witnessed by a pair of points in $\lne(t)\times {{H_2}}$, and this distance is a continuous function of $t$. At $t=t_0$, this distance is positive and $H_2$ is on the right of $\lne_0=\lne(t_0)$. Slide-turn this pointed supporting line around $H_1$ forward during $L$ seconds; that is, make a full turn around $\slcr{H_1}$. By continuity, in the chronological order listed below, there are
\begin{enumerate}
\item a last $t=t_1$ such that $H_2$ is still on the right of $\lne(t)$  (this $t_1$ exists, because it is the first value of $t$ where  $\dist{\lne(t)}{H_2}=0$), 
\item a first $t=t_2$ such that $H_2$ is on the left of $\lne(t)$, 
\item a last $t=t_3$ such that $H_2$ is still on the left of $\lne(t)$, 
\item a first $t=t_4$ such that $H_2$ is  on the right of $\lne(t)$. 
\end{enumerate}
In Figure~\ref{figfkzTn}, $h(t_i)=\pair{P(t_i)}{\dir{\lne(t_i)}}$ is represented by $\pair{P_i}{{\lne_i}}$. Clearly, $\lne_1,\dots,\lne_4$ is the list of all common supporting
lines and these lines are pairwise disjoint.
\end{proof}

\end{document}